\documentclass[graybox]{svmult}

\usepackage{type1cm}        
\usepackage{makeidx}         
\usepackage{graphicx}        
\usepackage{multicol}        
\usepackage[bottom]{footmisc}

\usepackage{newtxtext}       %
\usepackage[varvw]{newtxmath}       

\makeindex             

\newcommand{\R}{\mathbb{R}}
\newcommand{\C}{\mathbb{C}}
\newcommand{\N}{\mathbb{N}}
\newcommand{\cR}{\mathcal{R}}
\newcommand{\fG}{\mathfrak{G}}
\newtheorem{mainthm}{Theorem}

\begin{document}

\title*{On the Injectivity of STFT Phase Retrieval with super-exponential decaying window function}
\titlerunning{Injectivity of STFT Phase Retrieval}
\author{Shuang Guan and\\ Kasso A. Okoudjou}
\institute{Shuang Guan \at Department of Mathematics, Tufts University, Medford, MA 02155, USA \email{shuang.guan@tufts.edu}
\and Kasso A. Okoudjou \at Department of Mathematics, Tufts University, Medford, MA 02155, USA \email{kasso.okoudjou@tufts.edu}}
%
%
\maketitle

\abstract{We investigate the uniqueness of short-time Fourier transform phase retrieval problems in $L^2(\R)$. In particular, for underlying window functions whose Fourier transform decay faster than any exponential function, we derive  sufficient conditions on discrete sampling sets  for unique phase retrieval from the spectrogram. This result generalizes previous uniqueness guarantees on sampling sets for Gaussian windows.}
\section{Introduction}
\label{sec:1}
The phase retrieval problem refers to the task of recovering a signal from only intensity measurements under a given operator. This problem arises naturally in a wide range of applications in science and engineering, including X-ray crystallography \cite{millane1990phase}, coherent diffraction imaging \cite{walther1963question} and speech recognition \cite{picone1996fundamentals}. In recent years, the phase retrieval problem has attracted significant interests, involving tools from harmonic analysis, complex analysis, functional analysis and numerical optimization \cite{candes2015phase,alaifari2021phase,grohs2019stable}. Mathematically the phase retrieval problem is a nonlinear inverse problem that is typically ill-posed due to the natural ambiguity introduced by the loss of phase information. Over the years, addressing this subtle problem has led to a rich literature seeking to examine conditions for stability and feasibility of certain reconstruction algorithms, spanning both discrete and continuous settings.\\
\\
In this paper, we specifically study the phase retrieval problem in the setting of short-time Fourier transform (STFT). For a signal $f\in L^2(\R)$ and a window function $g\in L^2(\R)$, the STFT is defined as a function in $L^2(\R^2)$ given by:
$$V_gf(x,\omega) = \int_{\R} f(t)\overline{g(t-x)} e^{2\pi i \omega \cdot t}dt,  \quad (x, \omega) \in \R^2.$$
In many real-world applications, detectors can only capture the intensity, $|V_gf(x,\omega)|$, known as the spectrogram. The loss of phase information leads to the STFT phase retrieval problem: determine the unique signal $f$ up to a global sign from a sampling set of spectrogram $|V_gf(\Lambda)|$ for some $\Lambda \subset \R^{2}$.\\

While the STFT phase retrieval problem is ill-posed in general, it is known that uniqueness can still be guaranteed, up to a global phase factor, by imposing certain conditions on the window function $g$ and sampling on a sufficiently dense set $\Lambda$. However, the difficulty of this problem is two-fold: First, it was recently proved that if $\Lambda$ is an integer lattice, i.e, $\Lambda = A\mathbb{Z}^{2}$ for some $A \in GL(2,\R)$, then the uniqueness is never attained \cite{grohs2023injectivity}. 
The second difficulty is concerning the window function $g$, in the sense that if $V_gg(x,\omega)$ vanishes on a set of positive Lebesgue measure, then there's no established uniqueness result for general $f$.\\

Recently L. Liehr and P. Grohs \cite{grohs2024phaseless} showed that when $g$ is Gaussian, uniqueness is guaranteed for general $f \in L^2(\R)$ if the sampling set $\Lambda$ is a square-root lattice, i.e., $\Lambda = \{A(\pm\sqrt{n},\pm\sqrt{n}),n \in \N\}$ for some $A \in GL(2,\R)$. In what follows, we investigate how a weaker decay property of window function $g$ affects the injectivity result and determine the corresponding sampling set. More specifically, our main result can be stated as follows.

\begin{theorem}\label{MainTheorem}
    Let $g \in L^2(\R)$ with Fourier transform satisfies the super-exponential decay condition:
    \begin{equation*}
        |\hat{g}(\xi)| \leq C e^{-a|\xi|^m}
    \end{equation*}
    for some constants $m>1$, $C>0$, $a \in (1,\infty)$. Let $\tau_1, \tau_2 \in \R^+$ satisfy
    \begin{equation*}
        \tau_1 < \big(\frac{1}{(2\pi)^{m/(m-1)}(ma)^{-1/(m-1)}e}\big)^{\frac{m-1}{m}}, \quad \quad \quad \tau_2< (\frac{1}{ame})^{\frac{1}{m}}.
    \end{equation*}
    Define the set $\Lambda$ such that: 
    \begin{equation*}
        \Lambda =  \{(\pm\tau_1n^{\frac{m-1}{m}},\pm\tau_2n^{\frac{1}{m}}),n \in \N\}
    \end{equation*} 
    Then the following statements are equivalent for every $f, h \in L^2(\R):$
    \begin{enumerate}
        \item $|V_g(f)(\lambda)| = |V_gh(\lambda)|$ for every $\lambda \in \Lambda$.
        \item $f = e^{i \alpha}h$ for some $\alpha \in [0,2\pi).$
    \end{enumerate}
\end{theorem}

\begin{remark}
    A symmetric result holds for windows with super-exponential decay in the time domain. By the fundamental identity of the Short-time Fourier transform, we have $|V_g f (x, \omega)| = |V_{\hat{g}}\hat{f} (\omega, -x)|$. Consequently, if a window function $g$ satisfies the decay condition:
    \begin{equation*}
        |g(x)| \leq C e^{-a|x|^m}
    \end{equation*}
    for some constants $m>1$, $C>0$, $a \in (1,\infty)$, uniqueness is guaranteed by sampling on the set $\Lambda$ such that: 
    \begin{equation*}
        \Lambda =  \{(\pm\tau_2n^{\frac{1}{m}},\pm\tau_1n^{\frac{m-1}{m}}),n \in \N\}
    \end{equation*} where $\tau_1$ and $\tau_2$ satisfy the bounds in Theorem 1.
\end{remark}
\begin{figure}[h!]
    \centering
    \includegraphics[width=0.7\textwidth]{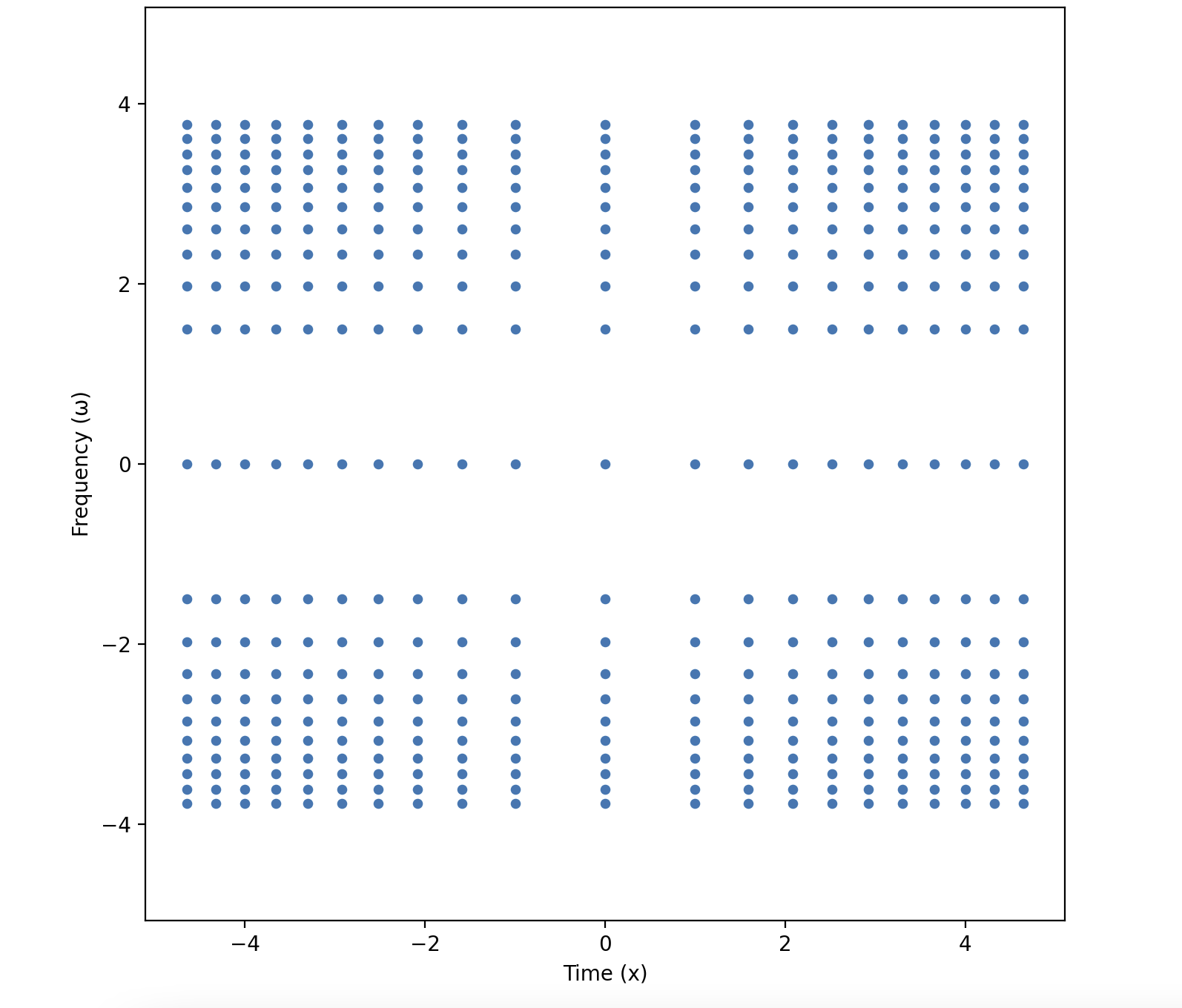}
    \caption{An illustration of sampling points $\Lambda$ in time-frequency plane for $g$ with Fourier decay parameter $m=\frac{3}{2}$ and $a=1$}
\end{figure}
The rest of the paper is organized as follows: In Section 2 we introduce some necessary preliminaries. Subsequently we prove Theorem~\ref{MainTheorem} in Section 3. To this end, we first establish a sufficient condition for a discrete set to be a uniqueness set for the relevant class of entire functions (Proposition 1) and use Proposition 2 to show that our result is in a sense "sharp". We then use these technical results to prove Theorem~\ref{MainTheorem}, which provides a concrete discrete sampling set $\Lambda$ that guarantees the unique recovery of $f$ up to a global phase.

\section{Preliminary}
\label{sec:2}
A central principle in Fourier analysis is the connection between the smoothness of a function and the decay rate of its Fourier transform at infinity. The Paley-Wiener theorem and its variants are the primary tools that precisely characterize this relationship, providing the foundation for extending functions into the complex domain based on the behavior of their Fourier transforms.  
\subsection{Paley-Wiener Type Theorem}
The Fourier transform of a Lebesgue integrable function $f$ on $\R$ is given by 
\begin{equation*}
    \mathcal{F}(f):=\hat{f} (\xi) = \int_{\R} f(t) e^{-2 \pi i t \cdot\xi} dt, \quad \xi \in \R.
\end{equation*} 

A standard density argument can be used to extend this definition so that  $\mathcal{F}$ is an isometry on $L^2(\R)$. The Fourier transform of a function $f \in L^2(\R)$ is said to be of \emph{super-exponential decay} if
\begin{equation}\label{FourierDecay}
    |\hat{f}(\xi)| \lesssim Ce^{-a|\xi|^m}    
\end{equation}
for some $m>1$, $C>0$ and $a \in (1,\infty)$. In this subsection we introduce a slightly modified version of Paley-Wiener theorem. In particular, we show that if the Fourier transform of a function $f$ decays super-exponential, then $f$ can 
be extended to an entire function whose order and type depend on the parameters $a$ and $m$. To establish this result, we make use of Morera’s Theorem, which we now recall. 
\begin{mainthm}\cite[Theorem 5.2]{stein2010complex}\label{Morera}
 If $\{ f_n \}_{n \in \N}$ is a sequence of holomorphic functions that converges uniformly to a function $f$ in every compact subset of $\Omega$, then $f$ is holomorphic in $\Omega$.
\end{mainthm}

The next result shows that if $f \in L^2(\R)$ satisfies~\eqref{FourierDecay}, then $f$ is entire.
\begin{lemma}\label{entire}
    If a function $f \in L^2(\R)$ has Fourier transform of super-exponential decay, then $f$ can be extended analytically to an entire function.
\end{lemma}
For completeness, we include the proof of this result which is based on arguments found in \cite[Chapter 3, Theorem 4.1]{stein1971introduction} and in \cite[Chapter 4, Problem 1]{stein2010complex}.\\

\begin{proof}
 Let $f\in L^2(\R)$ satisfy~\eqref{FourierDecay} for some $m>1, C>0, $ and $a\in (1, \infty)$. The $\hat{f} \in L^2(\R)$ from which we conclude that $\hat{f} \in L^1[-n,n]$ for each $n \in \N$.

    For each $n \in \N$ and $z=x+iy \in \C$, define 

\begin{equation*}
    F_n(z) =F_n(x+iy):= \int_{-n}^n \hat{f}(\xi)e^{2 \pi i \xi z}d\xi.
\end{equation*}
$\{F_n(z)\}_{n\in \N}$ forms a sequence of well-defined entire functions.

Next, for $z=x+iy \in \C$, observe that by~\eqref{FourierDecay} we have that 
$$|\hat{f}(\xi)e^{2 \pi i z \xi}|=|\hat{f}(\xi)e^{2 \pi i (x+iy) \xi}| \leq C' e^{-a|\xi|^m}e^{-2 \pi y |\xi|}.$$

Let $K\subset \C$ be a compact set. For each $z=x+iy\in K$, 
\begin{equation}\label{inverseFourierextension}
    F(z) = F(x+iy) = \int_{\R}\hat{f}(\xi) e^{2 \pi i (x+iy) \xi}d\xi
\end{equation}
is well defined. 


Moreover, for all $z \in K$, we have
\begin{equation*}
    |F_n(z) - F(z)|\leq C' \int_{|\xi|>n} e^{-a|\xi|^m}e^{-2 \pi y \xi}d\xi\leq C' \int_{|\xi|>n} e^{-a|\xi|^m}e^{-2 \pi R \xi}d\xi .
\end{equation*}
But the last expression tends to $0$ as $n$ approaches $\infty$. 
Hence the sequence $\{ F_n(z) \}$ converges uniformly to $F(z)$ for $z \in K$. Because $K$ is arbitrary, $F(z)$ is entire by Morera's Theorem. In addition, $\lim_{y\to 0} F(x+iy)=\int_{\R}\hat{f}(\xi)e^{2\pi i xi x} d\xi=f(x)$ in the $L^2$ sense. 

\end{proof}
\begin{remark} 
    For $m =1$, one can show that $f$ can be analytically extended to $\C$ up to a strip $\mathcal{S}_a$ consisting of complex numbers whose imaginary part is between $a$ and $-a$, but not necessarily entire. For example, consider function $f(x) = \frac{1}{x^2+1}$ which has Fourier transform $\hat{f}(\xi) = - 2\pi i e^{- 2\pi |\xi|}$ decays exponentially, yet its analytic extension $f(z)$ has a pole at $z= \pm i$ and is thus not entire.
\end{remark}

The following lemma is a corollary of the Paley-Wiener Theorem from the relation between $L^2-$ integrability on the real line and entire functions: 
\begin{corollary}\label{ImaginaryLine}
    Assume $f(z)$ is an entire function satisfying $|f(z)| \leq Ce^{A|z|^\rho}$ for some $A,C,\rho>0$ and the restriction of $f$ to the real line is square-integrable. Then $|f(x+iy)| \leq C' e^{-a|x|^{\rho}+b|y|^\rho}$ for some $a,b,C'>0$. 
\end{corollary}
\begin{proof}
    Notice that since $f(z)$ is entire, it has no poles for finite $z \in \C$ and since $f(x)$ is square-integrable on the real line, it holds that $\lim_{|x|\to \infty}f(x) \to 0$. Therefore $f(x)$ is bounded on the real line. The rest of the proof follows from a standard Phragmén–Lindelöf principle argument, see \cite[Page 124]{stein2010complex}.
\end{proof}
A natural question to investigate regarding an entire function is its order and type. Using various tools in complex analysis, the order and type of an entire function help to deduce various sampling theorems \cite{marco2003interpolating}. An entire function $f(z)$ is said to be of order $\rho \geq 0$ if 
\begin{equation}\label{OrderBound}
    \rho = \limsup_{r \to \infty} \frac{\log \log M_f(r)}{\log r} \geq 0, \quad \text{where}\; M_f(r) = \max_{|z|=r}|f(z)|.   
\end{equation}
Let $\rho$ be the order of an entire function $f$. The function is said to have a finite \textit{type} with respect to $\rho$ if for some $A,K>0$ the inequality
$$M_f(r) < Ke^{Ar^{\rho}}$$
is fulfilled. The greatest lower bound for those values of $A$ for which the inequality if fulfilled is called the \textit{type} $\tau$ corresponding to order $\rho$ of the function $f$. To help with calculating order and type of a given entire function, we introduce the following theorem:
\begin{mainthm}\cite[Theorem 2 and 3 in Lecture 1]{levin1996lectures}
    For an entire function $f(z) = \sum_{n=0}^{\infty}c_n z^n$, the order of the entire function is determined by the formula
    \begin{equation}\label{OrderCalculation}
        \rho = \limsup_{n \to \infty} \frac{n \log n}{\log (1/|c_n|)}.
    \end{equation}
    and the type $\tau$ with respect to order $\rho$ is determined by
    \begin{equation}\label{TypeCalculation}
        \tau = \frac{1}{e\rho} \limsup_{n \to\infty} (n \sqrt[n]{|c_n|^{\rho}}).
    \end{equation}
\end{mainthm}
Our next lemma gives an explicit calculation for the order and type of the entire function ~\eqref{inverseFourierextension}.

\begin{lemma}\label{OrderType}
Let $f \in L^2(\mathbb{R})$ be a function whose Fourier transform satisfies $|\hat{f}(\xi)| \le C e^{-a|\xi|^m}$ for some $m>1$ and $a>0$. Then its analytic continuation $f(z)$ is an entire function of type at most $\tau = \frac{m-1}{m}(2\pi)^{m/(m-1)}(am)^{-1/(m-1)}$ corresponding to maximum possible order $\rho = \frac{m}{m-1}$.
\end{lemma}
\begin{proof}
The analytic continuation of the function, $f(z)$, is given by the inverse Fourier transform:
\begin{equation*}
    f(z) = \int_{-\infty}^\infty \hat{f}(\xi) e^{2\pi i z \xi} d\xi
\end{equation*}
The Taylor coefficients of $f(z)$ are given by 
\begin{equation*}
c_n = \frac{f^{(n)}(0)}{n!} = \frac{(2\pi i)^n}{n!} \int_{-\infty}^\infty \xi^n \hat{f}(\xi) d\xi
\end{equation*}
We can now bound the magnitude of the coefficients using the given decay of $\hat{f}(\xi)$:
\begin{equation*}
|c_n| \le \frac{(2\pi)^n}{n!} \int_{-\infty}^\infty |\xi|^n |\hat{f}(\xi)| d\xi \le \frac{C(2\pi)^n}{n!} \int_{-\infty}^\infty |\xi|^n e^{-a|\xi|^m} d\xi
\end{equation*}
The integral can be evaluated using the Gamma function:
\begin{equation*}
    \int_{-\infty}^\infty |\xi|^n e^{-a|\xi|^m} d\xi = \frac{2}{m a^{(n+1)/m}} \Gamma\left(\frac{n+1}{m}\right). 
\end{equation*}
Using Stirling's formula for Gamma function, we obtain 
\begin{equation*}
 \Gamma(\frac{n+1}{m}) = \sqrt{\frac{2 \pi m}{n+1}}(\frac{n+1}{me})^{(n+1)/m}\big(1+ O(\frac{m}{n+1})\big).
\end{equation*}
and 
\begin{equation*}
    n! \approx \sqrt{2 \pi n}(\frac{n}{e})^n
\end{equation*}
Then, for some $C' >0$, we derive the following expression:
         \begin{align}\label{GeneralTaylorCoefficient}
              |c_n|
             &\lesssim \Big((2\pi)^n \cdot n^{-\frac{1}{2}-n}\cdot e^{n -\frac{n+1}{m}}\cdot a^{-\frac{n+1}{m}}\cdot m ^{-\frac{1}{2}-\frac{n+1}{m}}\cdot (n+1)^{\frac{n+1}{m}-\frac{1}{2}}\Big)\cdot\Big(1 + \frac{C'm}{n+1}\Big) \\
               &=  R(a,n,m)\cdot n^{-n}\cdot(n+1)^{\frac{n+1}{m}}\cdot\Big(1 + \frac{C'm}{n+1}\Big) \nonumber
        \end{align}
The term $R(a,n,m)$ collects all remaining exponential and polynomial factors. It follows that:
\begin{align*}
    \rho &= \limsup_{n\to\infty} \frac{n \log n}{-\log|c_n|} \lesssim \limsup_{n\to\infty} \frac{n\log n}{-\log R(a,m,n) + n\log n -\frac{n+1}{m}\log(n+1)- \log(1+\frac{C'm}{n+1})}\\
    &\lesssim \limsup_{n\to\infty} \frac{1}{\frac{-\log R(a,m,n)}{n\log n} + 1 -\frac{n+1}{nm}\frac{\log(n+1)}{\log n}- \frac{\log(1+\frac{C'm}{n+1})}{n\log n}}.
\end{align*}
A series of computations shows that $$\begin{cases} \lim_{n\to \infty}\frac{\log R(a,m,n)}{n\log n} =0\\ \lim_{n\to \infty}\frac{\log(1+ \frac{C'm}{n+1})}{n\log n}=0\end{cases}.$$  Therefore, the remaining terms are:
\begin{equation*}   
    \rho \leq \limsup_{n\to \infty}\frac{1}{1-\frac{n+1}{mn}\frac{\log(n+1)}{\log n}}= \frac{1}{1- \frac{1}{m}} = \frac{m}{m-1}.
\end{equation*}
Let $\tilde{\rho} = \frac{m}{m-1}$, and choose $A,R >0$ such that 
\begin{equation}\label{TypeWRTorder}
    \max_{|z|= r}|f(z)| < Re^{Ar^{\tilde{\rho}}}.    
\end{equation}
Now we compute the type $\tau$ of $f(z)$ with respect to order $\tilde{\rho}$, i.e., the greatest lower bound of $A$ in Equation~\eqref{TypeWRTorder}. We note that 

$$ |c_n|^{\tilde{\rho}/n} \lesssim \Big(n^{-\frac{1}{2}-n}(n+1)^{\frac{n+1}{m}-\frac{1}{2}}\Big)^{\tilde{\rho}/n}\, (e^{n - \frac{n+1}{m}})^{\tilde{\rho}/n}\, a^{-\frac{n+1}{m}\cdot \frac{\tilde{\rho}}{n}} \, m ^{(-\frac{1}{2}-\frac{n+1}{m})\cdot \frac{\tilde{\rho}}{n}}\, \Big(1 + \frac{C'm}{n+1}\Big)^{\tilde{\rho}/n} $$ But,

$$\Big(n^{-\frac{1}{2}-n}(n+1)^{\frac{n+1}{m}-\frac{1}{2}}\Big)^{\tilde{\rho}/n} = n^{-\frac{\tilde{\rho}}{2n}-\tilde{\rho}}(n+1)^{\frac{(n+1)\tilde{\rho}}{mn}-\frac{\tilde{\rho}}{2n}}=\Big(n^{-\tilde{\rho}}(n+1)^{\frac{(n+1)\tilde{\rho}}{mn}}\Big)\Big( n(n+1) \Big)^{^{-\frac{\tilde{\rho}}{2n}}}$$
Thus,

$$n|c_n|^{\tilde{\rho}/n}\lesssim  \Big(n^{1-\tilde{\rho}}(n+1)^{\frac{(n+1)\tilde{\rho}}{mn}}\Big)\Big( n(n+1) \Big)^{^{-\frac{\tilde{\rho}}{2n}}}(e^{n - \frac{n+1}{m}})^{\tilde{\rho}/n}\, a^{-\frac{n+1}{m}\cdot \frac{\tilde{\rho}}{n}} \, m ^{(-\frac{1}{2}-\frac{n+1}{m})\cdot \frac{\tilde{\rho}}{n}}\Big(1 + \frac{C'm}{n+1}\Big)^{\tilde{\rho}/n}$$

However, we see that 
$$\begin{cases}  \lim_{n\to \infty}(e^{n - \frac{n+1}{m}})^{\tilde{\rho}/n} = e^{\tilde{\rho}- \frac{\tilde{\rho}}{m}} = e^{\frac{m}{m-1} - \frac{1}{m-1}}=e\\
\lim_{n\to \infty} a^{-\frac{n+1}{m}\cdot \frac{\tilde{\rho}}{n}} = a^{-\tilde{\rho}/m}\\\lim_{n\to\infty}m ^{(-\frac{1}{2}-\frac{n+1}{m})\cdot \frac{\tilde{\rho}}{n}} = m^{-\tilde{\rho}/m}\\
\lim_{n\to \infty}\Big(1 + \frac{C'm}{n+1}\Big)^{\tilde{\rho}/n} =1\end{cases}.$$ Moreover,

$$\lim_{n\to \infty}\Big( n(n+1) \Big)^{^{-\frac{\tilde{\rho}}{2n}}} = 1 \;\quad \text{and} \lim_{n\to \infty}\Big(n^{1-\tilde{\rho}}(n+1)^{\frac{(n+1)\tilde{\rho}}{mn}}\Big)=\lim_{n\to \infty} n^{1-\tilde{\rho}+\frac{\tilde{\rho}}{m}}= 1$$
Therefore, using Equations~\eqref{TypeCalculation} and~\eqref{GeneralTaylorCoefficient}, combining the computations above we obtain: 
\begin{align*}
    \tau = \frac{1}{e\tilde{\rho}} \limsup_{n \to\infty} (n \sqrt[n]{|c_n|^{\tilde{\rho}}}) &= \frac{1}{e\tilde{\rho}}\lim_{n\to\infty}n\cdot(2\pi)^{\tilde{\rho}}\cdot a^{-\tilde{\rho}/m}\cdot m^{-\tilde{\rho}/m}\cdot n^{-1}\cdot e\\ 
    &= \frac{1}{\tilde{\rho}}  (2\pi)^{\tilde{\rho}}\cdot a^{-\tilde{\rho}/m}\cdot m^{-\tilde{\rho}/m}\\
    &= \frac{m-1}{m}(2\pi)^{m/(m-1)}(am)^{-1/(m-1)}.
\end{align*}
\end{proof}

Next, we introduce some necessary results of general theory of the growth of entire functions. For an entire function $f$, we define $n(r): [0,r) \to \N$ the number of zeros of $f$ within the region $|z|\leq r$. Jensen's formula describes the relation between zeros and an entire function to its growth.
\begin{mainthm}\cite[Formule 7]{jensen1899nouvel}\label{Jensen}
    Suppose $f$ is an entire function satisfying $f(0) \neq 0$. If $z_1, ..., z_n$ are the zeros of $f$ with $|z_j| \leq r$ (counting multiplicities), then 
    \begin{equation*}
        \frac{1}{2\pi} \int_0^{2\pi}\log|f(re^{i \theta})| d\theta  =\log|f(0)| + \sum_{j=1}^n \log(\frac{r}{|z_j|}) = \log|f(0)| + \int_0^r \frac{n(t)}{t}dt.
    \end{equation*}
\end{mainthm}
The next theorem, following the notation of this paper, gives a bound for the logarithm of an entire function when its zeros are all distributed on the positive real axis. 
\begin{mainthm}\cite[Chapter I, Theorem 25]{levin1964distribution}\label{Levin}
    Consider a sequence $\{ a_k \}$ of points on the positive real axis with $n(r)$ denoting the counting of $\# a_k$ inside $B_r(0)$ satisfying
    \begin{equation*}
       \Delta = \lim_{r \to \infty} \frac{n(r)}{r^{\rho(r)}}
    \end{equation*}
    for some $\rho>0$. If $\rho$ is not an integer and if 
    \begin{equation*}
        V(z) = \prod_{k=1}^{\infty}G(\frac{z}{a_k};p), \quad (p < \rho <p+1)
    \end{equation*}
    where $G(u;p) = (1-u)e^{u +u^2/2 + \cdots +u^p/p}$,
    then for $z = re^{i\theta}, 0 \leq \theta <2\pi$, the following limit holds uniformly for $\theta \in [0,2\pi)$:
    \begin{equation*}
        \lim_{r\to \infty}\frac{\ln V(re^{i\theta})}{r^{\rho}} = \frac{\pi \Delta}{\sin \pi \rho}e^{i \rho(\theta - \pi)},
    \end{equation*}
    while for integer $\rho$ we have 
    \begin{equation*}
        \lim_{r\to \infty}\frac{\ln V(re^{i\theta})}{r^{\rho}} = \pi \Delta e^{i \rho(\theta - \pi)},
    \end{equation*}
    where for taking logarithm of $G(u;p)$ the branch cut is taken along the positive real axis from $1$ to $\infty$ with $\text{arg} \; G(u;p) = -\pi$ along the upper side of the cut.
\end{mainthm}
\subsection{Uniqueness Result From Short-time Fourier Transform}
Consider a window function $g \in L^2(\R)$, define the \textit{Short-time Fourier Transform (STFT)}  $V_gf: \R^{2} \to \C$ with respect to $g$ by:
\begin{equation*}
    V_g f(x,\omega) := \int_{\R} f(t) \overline{g(t-x)} e^{-2 \pi i \omega \cdot t}dt.
\end{equation*}
It is known that if the window function $g$ is nowhere vanishing, then one can determine $f$ uniquely from $V_gf(x,\omega)$ from $(x,\omega)\in \R^{2}$. To carry out our proof, the following fundamental notations and established results in time-frequency analysis are required. For a function $f: \R \to \C$. the shift of $f$ by $\mu\in \R$ is defined by
\begin{equation*}
    T_\mu f(t): =f(t-\mu),
\end{equation*}
and the reflection map is
\begin{equation*}
    \cR f(t)= f(-t).
\end{equation*}
The short-time Fourier transform can be described as a composition of Fourier transform and shift, namely:
\begin{equation*}
    V_gf(x,\omega) = \mathcal{F}(f\overline{T_x g})(\omega).
\end{equation*}
For a real number $\zeta \in \R$ and $f: \R \to \C$ define the map $f_\zeta: \C \to \C$ by:
\begin{equation*} 
    f_\zeta(t):= (T_{\zeta} f(t))\overline{f(\overline{t}}).
\end{equation*}
If $f$ is assumed to be entire, then for every $\zeta \in \C$ and $t \in \C$ both maps
\begin{equation*}
    \zeta \to f_{\zeta}(t), \quad \quad t \to f_{\zeta}(t),
\end{equation*}
defines an entire function from $\C$ to $\C$.\\

The following lemma demonstrates that for the class of window functions we consider, if the spectrograms of two functions are identical across the entire time-frequency plane, then the functions themselves are equivalent up to a global phase
\begin{lemma}\label{spectrogram}
    Assume $0\neq  g \in L^2(\R)$ satisfies \eqref{FourierDecay} for some $C >0, a\in (1,\infty)$ and $m>1$. If $f,h \in L^2(\R)$ are such that $|V_gf(x,\omega)|^2=|V_gh(x,\omega)|^2$ for every $(x,\omega) \in \R^{2}$, then $f = e^{i\alpha}h$ for some $\alpha \in \R$.
\end{lemma}
\begin{proof}
    According to \cite[Corollary 3.3]{grohs2024phaseless}, it suffices to prove that for every fixed $\omega$, the set
    $$Z_\omega = \{ \xi \in \R: \mathcal{F}(\cR(g_\omega))(\xi) = 0\}$$
    has Lebesgue measure $0$. Notice that
    \begin{align}\label{WindowNoVanish}
    \mathcal{F}(\cR(g_\omega))(\xi) &= \mathcal{F}(g(-t-\omega)\overline{g(-t)})(\xi) \nonumber\\
    &= \int_\R g(-t-\omega)e^{-2 \pi i \xi t} dt  *\int_\R \overline{g(-t)}e^{-2 \pi i \xi t}dt\nonumber\\
    &=e^{2 \pi i \omega \xi}\hat{g}(-\xi)*\overline{\hat{g}(\xi)}\nonumber\\
    &= \int_{\R} e^{2\pi i \omega\eta}\hat{g}(- \eta))\overline{\hat{g}(\xi- \eta)}d\eta.
    \end{align}
    Using the fact that  $|\hat{g}(\xi)| \lesssim Ce^{-a|\xi|^m}$ and since $m > 1$, we obtain 
    $$|\hat{g}(- \eta))\overline{\hat{g}(\xi- \eta)}| \lesssim C' e^{-a'|\eta|^m}$$
    for some $C' >0$ and $a \in (1,\infty)$. It follows from Lemma \ref{entire} that ~\eqref{WindowNoVanish} defines a nonzero entire function. Hence $Z_\omega$ has Lebesgue measure $0$.
\end{proof}
While Lemma \ref{spectrogram} provides the desired uniqueness guarantee, it relies on knowing the spectrogram on the entire continuous plane. Our main theorem, however, depends on measurements from a discrete set. To bridge this gap, our strategy is to analytically continue the square of STFT spectrogram into the complex domain and use the deterministic properties of entire function. The following lemma proves not only that such an extension to entire functions is possible but also precisely characterizes the growth.
\begin{lemma}\label{STFTExtension}
    For every $f \in L^2(\R)$, and let the window function $g \in L^2(\mathbb{R})$ have a Fourier transform $\hat{g}(\xi)$ that satisfies \eqref{FourierDecay} for some constants $C, a > 0$ and $m>1$, then it holds that $|V_gf(x,\omega)|^2$ extends from $\R^{2}$ to entire functions $|V_gf(z,z')|^2$ in $\C^2$. Moreover, its growth is characterized by the following:
\begin{enumerate}
    \item With respect to the complex time variable $z$, it is of order at most $\rho_z = \frac{m}{m-1}$ and type at most $\tau_z = \frac{m-1}{m}(2\pi)^{m/(m-1)}(ma)^{-1/(m-1)}$.
    \item With respect to the complex frequency variable $z'$, it is of order at most $\rho_{z'} = m$ and and type at most $\tau_{z'} =a$.
\end{enumerate}
\end{lemma}
\begin{proof}
The extended STFT can be written as:
\begin{equation*}
    |V_gf(z,z')|^2 = \Bigg|\int_{-\infty}^{\infty} \overline{g(t -z)} e^{2\pi i z' t} f(t) dt\Bigg|^2 
\end{equation*}
For this integral to converge, we can analyze its magnitude using the Cauchy-Schwarz inequality. 
\begin{align*}
    \Bigg|\int_{-\infty}^{\infty} \overline{g(t -z)} e^{2\pi i z' t} f(t) dt\Bigg|^2
    &\le \left( \int_{-\infty}^{\infty} |\overline{g(t-z)}|^2 |e^{2\pi i z' t}|^2 dt \right) \left( \int_{-\infty}^{\infty} |f(t)|^2 dt \right)
\end{align*}

Fix $z'=x'+iy'\in \C$ and let $K\subset \C$ be compact. We prove that the previous integral converges uniformly for $z\in K$.

\begin{equation}\label{fixzprime}
    \int_{-\infty}^{\infty} |\overline{g(t-z)}|^2 |e^{2\pi iz't}|^2 dt = \int_{-\infty}^{\infty}|g(t-z)|^2e^{-4\pi y't}dt 
\end{equation}
From Corollary \ref{ImaginaryLine} we conclude that there exists positive real number $a,b,C$ such that 
\begin{equation*}
    |g(x+iy)| \leq Ce^{-a|x|^{\rho}+b|y|^{\rho}}.
\end{equation*}
which leads to
\begin{equation*}
    |g((t-x)+iy)|^2 \leq Ce^{-2a|t-x|^{\rho}+2b|y|^{\rho}}.
\end{equation*}
Therefore we have
\begin{align*}
\int_{-\infty}^{\infty}|g(t-z)|^2e^{-4\pi y't}dt \leq C\int_{-\infty}^{\infty}e^{-2a|t-x|^{\rho}+2b|y|^{\rho}-4\pi y't}dt<\infty
\end{align*}
It follows from Theorem A that \eqref{fixzprime} converges uniformly for $z=x+iy\in K$.

Similarly, fix $z=x+iy$ and consider $z' = x'+iy'\subset K'$, where $K'$ is compact. Then we have
\begin{equation}\label{fixz}
    \int_{-\infty}^{\infty} |\overline{g(t-z)}|^2 |e^{2\pi i z' \xi}|^2 dt \leq C\int_{-\infty}^{\infty} e^{-2a|t-x|^{\rho}+2b|y|^{\rho}-4\pi y't}dt.
\end{equation}
Since $z=x+iy$ is fixed, for $z' \in K'$ it follows that \eqref{fixz} converges for $z' \in K$. Since $|V_gf (z,z')|^2$ is entire in $z$ for every fixed $z'$ and is entire in $z'$ for every fixed $z$, By Hartog's Theorem on separate analyticity \cite[Theorem 1.2.5]{krantz2001function}, $|V_gf (z,z')|^2$ is holomorphic for all $(z,z') \in \C^2$. Following a statement from \cite[Page 68]{boas2011entire}, the order and type of an entire function is retained under integral, therefore the orders and types for $|V_gf (z,z')|^2$ can be computed directly using Lemma \ref{OrderType}. It is noteworthy that our computation for orders is compatible with the result from \cite[Lemma 3.4]{grohs2024phaseless} when $m = 2$.
\end{proof}
\section{Main Result}
Denote the ring of holomorphic functions by $\mathcal{O}(\C)$. Consider the collection of entire functions with $\rho>1$ and $a,b>0$:
$$\fG_{a,b,\rho} = \{ f \in \mathcal{O}(\C)\big| |f(x+iy)| \lesssim e^{-a|x|^{\rho}+b|y|^{\rho}} \}.$$
It is straightforward to see that $\fG_{a,b,\rho}$ is a linear space. The objective is to establish a uniqueness set so that whenever two functions $f, g \in \fG_{a,b,\rho}$ agree on the set, they agree \textit{everywhere}. Recent literature shows that such functions cannot be uniquely determined from their values on integer lattices \cite{alaifari2021phase}. For the purpose of notation, let $U \subset \mathcal{O(\C)}$ denote a linear function space. A set $\Lambda \subset \C$ is called a uniqueness set of $U$ if for any $f \in U$, it holds that 
$$(f( \lambda ) = 0 \; \text{for all} \; \lambda \in \Lambda) \Longrightarrow f = 0 \; \text{everywhere on } \: \C.$$
Since zero of any entire function can't be an accumulation point, any open set in $\C$ is trivially a uniqueness set. Without this property, determining discrete uniqueness sets of $\fG_{a,b,\rho}$ contained in $\R$ is a nontrivial task. We first propose a sufficient condition for such a set.\\

\begin{proposition}\label{UniquenessSet}
    Let $\rho > 1$ and $a,b \in \R^+$. Consider an increasing sequence $\{ \lambda_k \}$ such that
    \begin{equation*}
        \liminf_{k \to \infty} \frac{\lambda_k}{k^{1/\rho}} < \left( \frac{2}{b\rho e} \right)^{1/\rho} 
    \end{equation*}
    and a sequence $\Lambda := \{ \pm \lambda_k: k \in \N \} \subset \R$. Then $\Lambda$ is a uniqueness set for the class of entire functions in $\fG_{a,b,\rho}$.
\end{proposition}

\begin{proof}
    Suppose $f$ is a function in the given class and assume $f(0)=1$ without loss of generality. We begin with Jensen's formula (Theorem \ref{Jensen}):
    \begin{equation*}
        \int_0^r \frac{n(t)}{t}dt = \frac{1}{2\pi} \int_0^{2\pi}\log |f(re^{i\theta})|d\theta. 
    \end{equation*}
    The growth condition on $f(z)$ implies that for $z = re^{i\theta} = r\cos\theta + ir\sin\theta$:
    \begin{equation*}
        \log|f(re^{i\theta})| \le \log C - a|r\cos\theta|^\rho + b|r\sin\theta|^\rho = \log C + r^\rho(-a|\cos\theta|^\rho + b|\sin\theta|^\rho).
    \end{equation*}
    To find an upper bound for the integral, we maximize the angular term $g(\theta) = -a|\cos\theta|^\rho + b|\sin\theta|^\rho$. This term is maximized when $|\sin\theta|=1$ and $|\cos\theta|=0$, which occurs at $\theta = \pm \pi/2$. The maximum value is $g(\pi/2) = -a(0)^\rho + b(1)^\rho = b$. Then we have
    \begin{equation*}
        \int_0^r \frac{n(t)}{t}dt \le \log C + br^\rho.
    \end{equation*}
    Since the zero-counting function $n(t)$ is non-decreasing, for any $s>1$ we can estimate $n(r)$:
    \begin{equation*}
        n(r) \log s \le \int_r^{sr} \frac{n(t)}{t}dt \le \int_0^{sr} \frac{n(t)}{t}dt \le \log C + b(sr)^\rho = \log C + bs^\rho r^\rho. 
    \end{equation*}
    Where in the inequalities above, we used the fact that $n(sr)\geq n(r) \geq n(0)$ are fixed number and therefore 
    \begin{equation*}
        n(sr)\log(sr)-n(r)\log(r)\geq n(r)\big(\log(sr)- \log(r)\big) = n(r)\log(s).
    \end{equation*}
    Now, assume $f$ vanishes on $\Lambda$ but is not identically zero. We choose $r = \lambda_k$, which implies that the number of zeros $n(r) = \#(\lambda_k) \ge 2k$. Substituting this into our inequality gives
    \begin{equation*}
        2k \log s \le \log C + bs^\rho \lambda_k^\rho.
    \end{equation*}
    We rearrange the terms to find a lower bound for the limit infimum of the zero distribution:
    \begin{equation*}
        \frac{2 \log s}{bs^\rho} - \frac{\log C}{bs^\rho k} \le \frac{\lambda_k^\rho}{k}.
    \end{equation*}
    Taking the limit inferior as $k \to \infty$, the term containing $\log C$ vanishes, and we get
    \begin{equation*}
        \frac{2 \log s}{bs^\rho} \le \liminf_{k \to \infty} \frac{\lambda_k^\rho}{k} = \left( \liminf_{k \to \infty} \frac{\lambda_k}{k^{1/\rho}} \right)^\rho.
    \end{equation*}
    This inequality holds for any $s>1$. To obtain the tightest possible bound, we must maximize the left-hand side with respect to $s$. Let $g(s) = \frac{\log s}{s^\rho}$. The derivative $g'(s) = s^{\rho-1}(1 - \rho\log s) / s^{2\rho}$ is zero when $1 - \rho\log s = 0$, which occurs at $s = e^{1/\rho}$. \\
    The maximum value of $g(s)$ is therefore:
    \begin{equation*}
        g(e^{1/\rho}) = \frac{\log(e^{1/\rho})}{(e^{1/\rho})^\rho} = \frac{1/\rho}{e} = \frac{1}{\rho e}.
    \end{equation*}
    Substituting this maximum value back into our inequality gives the necessary condition for a non-zero function to exist:
    \begin{equation*}
        \frac{2}{b} \left( \frac{1}{\rho e} \right) \le \left( \liminf_{k \to \infty} \frac{\lambda_k}{k^{1/\rho}} \right)^\rho.
    \end{equation*}
    Taking the $\rho$-th root of both sides, we find the final lower bound:
    \begin{equation*}
        \left( \frac{2}{b\rho e} \right)^{1/\rho} \le \liminf_{k \to \infty} \frac{\lambda_k}{k^{1/\rho}}.
    \end{equation*}
    Hence, if
    \begin{equation*}
        \left( \frac{2}{b\rho e} \right)^{1/\rho} > \liminf_{k \to \infty} \frac{\lambda_k}{k^{1/\rho}}.
    \end{equation*}
    would give a contradiction and $f$ must be $0$ identically.
\end{proof}
Next we introduce a lower bound of the condition on the zero set for $\fG_{a,b,\rho}$. The way we prove this is to construct a nontrivial entire function with prescribed zeros and show that this function is indeed in $\fG_{a,b,\rho}$.

\begin{proposition}\label{Existence}
    Let $\rho > 1$ and $a, b \in \R^+$. Consider an increasing sequence $\{ \lambda_k \}, k \in \N$ and $\Lambda := \{ \pm \lambda_k: k \in \N \}$. If
    \begin{equation*}
        \liminf_{k \to \infty} \frac{\lambda_k}{k^{1/\rho}} > C_\rho, 
    \end{equation*}
    where $C_\rho$ is defined as some positive number satisfying
    \begin{equation*}
        C_\rho = 
        \begin{cases} 
            \left( \frac{\pi}{b \sin(\pi\rho/2)} \right)^{1/\rho} & \text{if } \rho/2 \text{ is not an integer} \\
            \left( \frac{\pi}{b} \right)^{1/\rho} & \text{if } \rho/2 \text{ is an integer}
        \end{cases}
    \end{equation*}
    then $\Lambda$ is \textbf{not} a uniqueness set for  functions in $\fG_{a,b,\rho}$.
\end{proposition}
\begin{proof}
    Assume $\rho/2$ is not an integer. Let $V(w)$ be the canonical product of genus $p = \lfloor \rho/2 \rfloor$ constructed from the zeros with $m \in \N_0$ and $\omega_k = \lambda_k^2$:
    \begin{equation*}
        V(\omega) = \omega^m\prod_{k=1}^{\infty} G(\frac{\omega}{\omega_k},p)
    \end{equation*}
    where
    \begin{equation*}
        G(u,p)=(1-u)e^{u+ u^2/2 +\cdots +u^p/p}.
    \end{equation*}
    The growth of the canonical product $V(w)$ can be determined from the distribution of its zeros. To build this connection, we first define a quantitative description of zeros $\Delta_V$ for $V(\omega)$ by 
    \begin{equation*}
        \Delta_V = \lim_{r \to \infty}\frac{n_V(r)}{r^{\rho/2}}
    \end{equation*}
    where $n_G(r)$ is the zero-counting function for $V$ inside $B_r(0) \subseteq \C$. Since $V$ is an entire function of finite order and its zeros are assumed to lie on the positive ray, it is of \textit{completely regular growth} and the limit $\Delta_V$ always exists. The condition on $\Lambda$ indicates that for any $\epsilon > 0$, there exists an integer $K$ such that for all $k>K$:
    \begin{equation*}
        \lambda_k > (C_\rho + \epsilon)k^{1/\rho}.
    \end{equation*}
    The zeros of $V(\omega)$ are precisely $\omega_k = \lambda_k^2$, so for $k>K$, we have $\omega_k > (C_\rho + \epsilon)^2 k^{2/\rho}$, therefore, consider zeros inside $B_r(0)$, we have
    \begin{equation*}
        (C_\rho + \epsilon)^2 k^{2/\rho} < \omega_k \le r \implies k < \left(\frac{r}{(C_\rho + \epsilon)^2}\right)^{\rho/2} = \frac{r^{\rho/2}}{(C_\rho + \epsilon)^\rho}.
    \end{equation*}
    Hence the total counting of zeros for any $r \in \R^+$ must satisfy this bound:
    \begin{equation*}
        n_V(r) \le K + \frac{r^{\rho/2}}{(C_\rho+\epsilon)^\rho},
    \end{equation*}
    where on the right hand side, the $K$ term comes from counting $\omega_i, i \in \{ 1, \cdots, K\}$ and the later term comes from the limit behavior. Finally we conclude that 
    \begin{equation*}
        \Delta_V = \lim_{r \to \infty}\frac{n_V(r)}{r^{\rho/2}}\leq \frac{1}{(C_\rho+\epsilon)^\rho}<\frac{1}{C_\rho^\rho}.
    \end{equation*}
    Using Theorem \ref{Levin} we obtain 
    \begin{equation}\label{LevinFormula}
        \lim_{r\to \infty}\frac{|\ln V(r e^{i \theta})|}{r^{\rho/2}} = \frac{\pi \Delta_V}{|\sin \pi \rho/2|}.  
    \end{equation}
    Notice that for finite $\omega \in \C$, $V(\omega)$ has no pole, we can always find a fixed real number $C$ such that $|G(x+iy)| \leq C e^{-a|x|^{\rho/2}+b|y|^{\rho/2}}$. Therefore to prove our argument, it suffices to prove 
    \begin{equation*}
        \lim_{r\to \infty}\frac{\ln |V(r e^{i \theta})|}{r^{\rho/2}}< b.
    \end{equation*}
    Substitute the value of $\Delta_V$ into Equation \eqref{LevinFormula}, we obtain
    \begin{equation*}
        \lim_{r\to \infty}\frac{|\ln V(r e^{i \theta})|}{r^{\rho/2}} < \frac{\pi}{|\sin \pi \rho/2|}\cdot\frac{b\sin(\pi \rho/2)}{\pi} =b
    \end{equation*}
    notice that we assumed the zeros are all on the positive real axis, hence we automatically have $\sin \pi \rho/2$ is positive. Therefore we conclude that 
    \begin{equation*}
        \lim_{r \to \infty}\frac{\ln |V(r e^{i \theta})|}{r^{\rho/2}} = \lim_{r \to \infty}\frac{\Re(\ln V(r e^{i \theta}))}{r^{\rho/2}} \leq \lim_{r \to \infty}\frac{|\ln V(r e^{i \theta})|}{r^{\rho/2}}<b.
    \end{equation*}
    The proof follows similarly for $\rho/2$ is an integer. Finally, consider $F(z) = V(z^2)$, then the nontrivial function $F$ vanishes on $\Lambda$ and satisfies the growth order and type. 
\end{proof}
We are now ready to prove Theorem \ref{MainTheorem}. 
\begin{proof}
(2) $\implies$ (1):By the definition and linearity of the STFT:
\begin{equation*}
    V_g f(x, \omega) = \int_{\R} (e^{i\alpha} h(t)) \overline{g(t-x)} e^{2\pi i \omega t} dt = e^{i\alpha} V_g h(x, \omega) 
\end{equation*}
Taking the magnitude, $|V_g f(x, \omega)| = |e^{i\alpha}| |V_g h(x, \omega)| = |V_g h(x, \omega)|$. This holds for all $(x, \omega) \in \R^2$, and thus for all $\lambda \in \Lambda$.
\\
(1) $\implies$ (2): Using Lemma \ref{STFTExtension} we can obtain the extended $H_f(z,z') = |V_gf(z,z')|^2$ and $H_h(z,z') = |V_gh(z,z')|^2$ as two entire functions in $z$ and $z'$ with known order and type in each variable. Denote $H(z,z') = H_f(z,z') -H_h(z,z')$. From Condition (1) we know that $H(x_k, \omega_k) =0$ for all $(x_k, \omega_k) \in \Lambda$. Consider the one-variable collection of entire functions $h_k(z) = H(z, \omega_k)$, each $h_k(z)$ is zero on $z = x_k$ for $(x_k,\omega_k)\in\Lambda $. Using the order and type computed in Lemma \ref{STFTExtension} and by the assumption on $\tau_1$, we obtain $\liminf_{j \to \infty}\frac{x_j}{j^{1/\rho_x}}<(\frac{2}{b_x\rho_x e})^{1/\rho_x}$ where $\rho_x,b_x$ are order and type of $h_k(z)$. Therefore Proposition $1$ indicates that $h_k(z)= H(z,\omega_k) = H_f(z,\omega_k)-H_h(z,\omega_k)=0$ for all $z \in \C$. Now fix arbitrary $z_0 \in \C$ and consider the function $\varphi_0(z') = H(z_0,z')$. By Lemma \ref{STFTExtension}, $\varphi_0(z')$ is an entire function of order $m$ and type $a$. By the assumption on $\tau_2$, Proposition $1$ indicates that the collection of $\omega_i$ forms a uniqueness set for $\varphi_0(z')$. Since $z_0\in \C$ is arbitrary, we have shown $H(z,z')\equiv 0$ for all $(z,z') \in \C^2$, which implies the restrictions of $z, z'$ to their real parts $x$ and $\omega$ satisfy $|V_g f(x, \omega)|^2 = |V_g h(x, \omega)|^2$. By Lemma \ref{spectrogram}, this is sufficient to conclude that $f = e^{i \alpha}h$ for some $\alpha \in \R$.
\end{proof}
\begin{remark}
    Proposition 2 tells that our result is sharp in the sense that the sufficient requirement on $\tau_1,\tau_2$ is not arbitrary.
\end{remark}
\begin{acknowledgement}
This work was partially supported by the National Science Foundation, grants DMS-2205771 and DMS-2309652. The authors thank the reviewers for their insightful and helpful comments.
\end{acknowledgement}

\ethics{Competing Interests}{The authors declare no conflict of interest.}

%
%
%


\end{document}